\documentclass[12pt,reqno]{amsart}

\usepackage{amsfonts}
\usepackage{amscd}
\usepackage{amssymb}
\usepackage{amsfonts}
\usepackage{amscd}
 \usepackage{amsmath} 
\usepackage{mathrsfs}
\usepackage{enumerate}
\usepackage{float} 
\usepackage[pdftex]{graphicx}
\allowdisplaybreaks

\usepackage{hyperref}

\date{\today}

\newtheorem{thm}{Theorem}[section]

\newtheorem{lem}[thm]{Lemma}

\theoremstyle{definition}
\newtheorem{defn}[thm]{Definitions}
\theoremstyle{remark}

\numberwithin{equation}{section}

\newcommand{\R}{\mathbb R}

\newcommand{\N}{\mathbb N}

\newcommand{\ml}{\mathcal{L}}
\newcommand{\g}{\mathfrak{g}}
\newcommand{\la}{\lambda}
\newcommand{\C}{{\mathbb C}}

\newcommand{\La}{\langle}
\newcommand{\Ra}{\rangle}

\newcommand{\mz}{\mathfrak{z}}
\newcommand{\p}{\mathfrak{p}}
\newcommand{\q}{\mathfrak{q}}
\newcommand{\F}{\mathcal{F}}
\DeclareMathOperator{\pf}{Pf}

\usepackage{color}
\newcommand{\blue}[1]{\textcolor{blue}{#1}}

\title[$L^p$-Uncertainty inequalities]
{$L^p$- Heisenberg--Pauli--Weyl uncertainty inequalities on certain two-step nilpotent Lie groups}

\author[Ganguly and Sarkar]{ Pritam Ganguly and Jayanta Sarkar}

\address[Pritam Ganguly]{{ Stat-Math Unit,
Indian Statistical Institute, Kolkata,
BT Road, Baranagar, Kolkata 700108}}
\email{pritam1995.pg@gmail.com}

\address[Jayanta Sarkar]{School of Mathematics \\
Indian Institute of Science Education and Research Thiruvananthapuram \\
Maruthamala P. O, Vithura\\ Kerala 695551, India.}
\email{ jayantasarkarmath@gmail.com, jayantasarkar@iisertvm.ac.in}
 \date{}

\begin{document}

\maketitle

\begin{abstract} 
This article presents the $L^p$-Heisenberg--Pauli--Weyl uncertainty inequality for the group Fourier transform on a class of two-step nilpotent Lie groups, specifically the M\'etivier groups. This inequality quantitatively demonstrates that on M\'etivier groups, a nonzero function and its group Fourier transform cannot both be sharply localized. The proof primarily relies on utilizing the dilation structure inherent to two-step nilpotent Lie groups and estimating the Schatten class norms of the group Fourier transform. The inequality we establish is new, even in the simplest case of Heisenberg groups. Our result significantly sharpens all previously known $L^p$-Heisenberg--Pauli--Weyl uncertainty inequalities for $1 \le p < 2$ on M\'etivier groups.

\end{abstract}

\section{Introduction and main results}
Nearly a century ago, in 1927, Heisenberg \cite{Heis} first introduced the idea that the position and momentum of a particle cannot be precisely determined at the same time in any quantum-mechanical state. Later, Kennard \cite{Ken} and Weyl \cite{Weyl} provided a more rigorous mathematical foundation for this concept, attributing it to Pauli. Nowadays, the mathematical community refers to this result as the Heisenberg--Pauli--Weyl uncertainty inequality, which, in its most general form, states that for every $\alpha,\beta>0$ and $f\in L^2(\R^n)$,
\begin{equation}
    \label{HPW-R}
    \|f\|_{L^2(\R^n)}^{\alpha+\beta}\leq C(\alpha,\beta,n) \left(\int_{\R^n}\|x\|^{2\alpha}|f(x)|^2~dx\right)^{\frac{\beta}{2}}\left(\int_{\R^n}\|\xi\|^{2\beta}|\widehat{f}(\xi)|^2~d\xi\right)^{\frac{\alpha}{2}}
\end{equation}
where $\widehat{f}$ stands for the Fourier transform of $f$, and $\|\cdot\|$ is the standard Euclidean norm.  Beyond its fundamental role in quantum physics and signal processing, this inequality offers profound mathematical insights, particularly through a striking property of Fourier pairs: \textit{a function and its Fourier transform cannot both be sharply localized}. This concept, known as the uncertainty principle in harmonic analysis, has captivated mathematicians for decades, inspiring various formulations across different mathematical contexts. Notably, Hardy introduced a qualitative version of this principle (see \cite[p. 227]{FSi}), whereas \eqref{HPW-R} represents an important quantitative counterpart. For a comprehensive overview of the history of various uncertainty principles in harmonic analysis, as well as the significance of this inequality and its generalizations to other \( L^p \) norms, we refer the reader to the survey article by Folland and Sitaram \cite{FSi}. We refer the reader to Section \ref{prelim} for any undefined notions and symbols in this section.  

As is often the case in mathematics, numerous researchers, guided by mathematical intuition, have explored various generalizations of the Heisenberg--Pauli--Weyl (HPW) inequality. To establish a formal foundation for this discussion, we first note that, in view of the Plancherel formula, the inequality can be reformulated as  
\begin{equation}
    \label{HPW-Rn}
    \|f\|_2^{\alpha+\beta}\leq C(\alpha,\beta,n) \big\|\|\cdot\|^\alpha f\big\|^\beta_2 ~\|(-\Delta_{\R^n})^{\frac{\beta}2}f\|^\alpha_2,
\end{equation}
where $\Delta_{\R^n}$ is the Laplacian on $\R^n$ and the operator $(-\Delta_{\R^n})^{\frac{\beta}2}$ is defined by
\begin{equation}
    \label{FT-of-laplace power}
    \widehat{(-\Delta_{\R^n})^{\frac{\beta}2}f}(\xi)=\|\xi\|^{\beta}\widehat{f}(\xi)\quad\quad (f\in C_c^{\infty}(\R^n)) .
\end{equation} 
The HPW inequality, expressed in this form, does not involve the Fourier transform, thereby allowing for generalizations to broader settings, where the Laplacian is replaced by a positive, self-adjoint operator and $\|\cdot\|$ is substituted with a suitable distance function. 

The literature in this field is so extensive that it is not feasible to refer to every relevant paper in this direction.  However, we will highlight key works that have significantly influenced this area of research and help shape the path we intend to take.  Thangavelu \cite{Th} was the first to establish an analogue of the HPW inequality \eqref{HPW-Rn} for $\alpha = \beta = 1$, replacing the Laplacian with Hermite and special Hermite operators, as well as the sub-Laplacian on Heisenberg groups. Later, Sitaram--Sundari--Thangavelu \cite{SST} extended the result for the Heisenberg group to the range $0 < \alpha = \beta < Q/2$. Xiao--He \cite{XH} further generalized the inequality by proving it for all $\alpha, \beta > 0$ on Heisenberg groups. Ray \cite[Theorem 4.1]{R} established an analogue of \eqref{HPW-Rn} on two-step nilpotent Lie groups for $\alpha = \beta = 1$, where the Laplacian is replaced by the sub-Laplacian of the group. Ciatti--Ricci--Sundari \cite{CRS} extended these results to two-step nilpotent Lie groups. Later, the same authors further extended \eqref{HPW-Rn} to a broader setting of Lie groups with polynomial volume growth, replacing the Laplacian with a H\"ormander-type sub-Laplacian (see \cite[Theorem 2.1]{CRS1}). Finally, Ciatti--Cowling--Ricci established the following more general version of \eqref{HPW-Rn} by replacing the $L^2$ norm with $L^p$ norms in the context of stratified Lie groups.
\begin{thm}[\cite{CCR}]
\label{Cowling HPW}
    Let $G$ be a stratified Lie group equipped with a sub-Laplacian $L$, and let $|\cdot|$ denote a homogeneous norm. Assume that $\beta,\:\delta\in(0,\infty)$, $p,\:r\in(1,\infty)$, $s \geq 1$, and that  
    $$
    \frac{\beta+\delta}{p} = \frac{\beta}{r} + \frac{\delta}{s}.
    $$
Then, there exists a constant $C>0$ such that for any $f \in C_c^\infty(G)$, the following holds: 
 \begin{equation}
    \label{CCR-HPW-stratified}
    \|f\|_p\leq C\||\cdot|^\beta~f\|^{\frac{\delta}{\beta+\delta}}_s\|L^{\delta/2}f\|_r^{\frac{\beta}{\beta+\delta}}.
\end{equation}
\end{thm}
By utilizing homogeneity, one can verify that for the inequality \eqref{CCR-HPW-stratified} to hold, the condition on $(p, r, s)$ must necessarily be satisfied. This makes the theorem the most general result in its class. However, it is important to highlight the work of Martini \cite{M}, who extended the $L^2$-HPW inequality to broader settings. His results not only encompass the groups discussed above but also apply to spaces with exponential volume growth. Furthermore, employing a completely different approach based on the isoperimetric inequality, Dall'Ara and Trevisan \cite{DT} established an $L^p$ version of \eqref{HPW-Rn} for a fairly general class of spaces, including non-compact Riemannian symmetric spaces, with $\beta = 1$ and $(-\Delta_{\R^n})^{1/2}$ replaced by an invariant gradient. Later, Martín--Milman \cite{MM} extended the $p=$ case to general metric measure spaces without assuming any underlying group structure. Notably, their approach remains valid in the Gaussian setting as well.

As evident from the above discussion, the \( L^p \)-HPW inequality has only been studied beyond Euclidean spaces by incorporating powers of positive self-adjoint operators, specific to each setting on the right-hand side.  

In the context of the Fourier transform on $\mathbb{R}^n$, Steinerberger~\cite[Theorem~1]{St} recently investigated an $L^1$ variant of~\eqref{HPW-R}. Subsequently, Xiao \cite[p. 273]{X} extended his result to the full range \( 1 < p < \infty \). More precisely, they proved the following.
\begin{thm}
    \label{Stein-Xiao}
    Let $1\leq p<\infty.$ Suppose that $\alpha,\beta\in(0,\infty)$ are such that $\beta> n(1/p-1/2)$ whenever $1\leq p \leq2$, and $\alpha<n/p'$ whenever $p>2.$ Then, for any $f\in L^p(\R^n)$, one has 
    \begin{equation}
    \label{Lp-HPW-Rn}
         \|f\|_p^{\alpha+\beta}\leq C(\alpha,\beta,n,p) \left(\int_{\R^n}(\|x\|^{\alpha}|f(x)|)^p~dx\right)^{\frac{\beta}p}\left(\int_{\R^n}(\|\xi\|^{\beta}|\widehat{f}(\xi)|)^{p'}~d\xi\right)^{\frac{\alpha}{p'}}.
    \end{equation}
\end{thm}
  Notice that, by the Plancherel formula, the HPW inequalities \eqref{HPW-R} and \eqref{HPW-Rn} are equivalent; however, their \( L^p \) counterparts for \( p \neq 2 \) are not. In fact, by the Hausdorff--Young inequality, the Fourier transform version \eqref{Lp-HPW-Rn} is sharper whenever \( 1 \leq p < 2 \). However, by the dual Hausdorff--Young inequality, Xiao's result for $p>2$ follows from Theorem~\ref{Cowling HPW}.  This version has garnered significant attention in recent years; see, for example, \cite{FX}, \cite{CDM}. However, beyond Euclidean spaces—such as in two-step nilpotent Lie groups—the Fourier transform is operator-valued, making it considerably more challenging to formulate and establish a precise analogue of \eqref{Lp-HPW-Rn} in such settings. In fact, the sharper version involving the Fourier transform presented in the theorem above has drawn our interest and serves as the primary motivation for investigating the possibility of an \( L^p \) version of the HPW inequality beyond Euclidean spaces.

In this article, we aim to establish an $L^p$ version of the HPW uncertainty inequality for the group Fourier transform on M\'etivier groups. To present our results, we introduce the necessary notation with minimal explanation, deferring a more detailed discussion to Section \ref{prelim}. Let $G$ be a M\'etivier group with Lie algebra $\mathfrak{g} = \mathfrak{g}_1 \oplus \mathfrak{g}_2$, where  $[\g_1,\g_1]=\g_2$, and so $\g_2$ is contained  the center   of $\g$.  It is known that there exists a Zariski-open subset $\Lambda \subset \mathfrak{g}_2^*\setminus\{0\}$ of full measure, which parametrizes the irreducible unitary representations of $G$ relevant to the Plancherel measure. Moreover, for each $\lambda \in \Lambda$, the corresponding representation $\pi_{\lambda}$ is realized on $L^2(\mathfrak{p}_{\lambda})$, where $\mathfrak{p}_{\lambda}$ is a subspace of $\mathfrak{g}_1$. For a suitable function \( f \) on \( G \), the group Fourier transform \( \mathcal{F}(f) \) is an operator-valued function on \( \Lambda \), where for each \( \lambda \in \Lambda \), \( \mathcal{F}(f)(\lambda) \) is a bounded operator on \( L^2(\mathfrak{p}_{\lambda}) \). Since the Fourier transform is operator-valued, deriving an analogue of the Fourier-side expression on the right-hand side of \eqref{Lp-HPW-Rn} requires careful observation. In this setting, the sub-Laplacian \( \mathcal{L} \) serves as the analogue of the Laplacian, and its Fourier transform satisfies  
$$\mathcal{F}(\mathcal{L}^{\beta/2} f)(\lambda) = \mathcal{F}(f)(\lambda) H(\eta(\lambda))^{\beta/2},$$  
where \( H(\eta(\lambda)) \) denotes the generalized scaled Hermite operator (see \eqref{scaledher}). As a result, the term \( \|\xi\|^{\beta} \) is naturally replaced by \( H(\eta(\lambda))^{\beta/2} \) in this framework. A similar approach was employed in the special case of Heisenberg groups in \cite{SST} to address the \( L^2 \)-case. This insight naturally leads us to the following precise analogue of Theorem \ref{Stein-Xiao}. 
\begin{thm}
\label{main result}
    Let $1\leq p< 2$. Assume that $\gamma>0,~\beta>Q(1/p-1/2)$, and $f\in L^p(G).$ Then for $p=1$, we have
    \begin{align*}
       \|f\|^{\gamma+\beta}_1\leq C(\gamma,\beta, n,k)\left(\int_{G}|x|^\gamma |f(x)|dx\right)^\beta~\left( \sup_{\lambda\in \Lambda} \|\mathcal{F}(f)(\lambda) H(\eta(\lambda))^{\frac\beta2}\|_{\text{op}}\right)^\gamma. 
    \end{align*}
    and for $p>1$, we have
    \begin{align*}
        \|f\|_p^{\gamma+\beta}\leq C(\gamma,\beta, n,k,p) \||\cdot|^{\gamma}f\|_p^\beta \left( \int_{\Lambda} \| \F (f)(\lambda)H(\eta(\lambda))^{\frac{\beta}{2}}\|^{p'}_{S_{p'}(L^2(\mathfrak{p}_\la))} |\pf(\lambda)|~d\lambda\right)^{\frac{\gamma}{p'}},
    \end{align*}
     
\end{thm}
Here, $Q$ is the homogeneous dimension of $G$, $|\cdot|$ denotes a homogeneous norm on $G$, and $\|\cdot\|_{S_{p'}(L^2(\mathfrak{p}_\la))}$ denotes the Schatten $p'$-norm on the Hilbert space $L^2(\mathfrak{p}_\la)$ (see Appendix \ref{schatt}).

This inequality demonstrates a quantitative uncertainty principle in the realm of two-step nilpotent groups, illustrating the limitation of simultaneously localizing a function in some weighted $L^p$ norm and its operator-valued Fourier transform in an appropriate Schatten norm. The key novelty of our proof lies in extending spectral techniques for the operator-valued Fourier transform beyond $L^2(G)$, combined with the use of the inherent dilation structure present in the group. The dilation structure in the Euclidean setting is isotropic, whereas in \( G \), it is non-isotropic. As a result, the appearance of the homogeneous dimension \( Q \) in Theorem \ref{main result}, rather than the topological dimension as in the Euclidean case, is a natural consequence.

Now, we compare our result with previously established analogues of the HPW inequality in the context of nilpotent Lie groups. Our primary focus is on the works of Ciatti--Cowling--Ricci \cite{CCR} and Dall’Ara--Trevisan\cite{DT}, as these are the only two studies that address the $L^p$ case.  

In light of the Hausdorff--Young inequality, our result strengthens Theorem~\ref{Cowling HPW} in the range $1<p<2$ for M\'etivier groups. Additionally, our case for $p = 1$ is entirely novel, as \cite{CCR} does not address the case $p = 1$. 

On the other hand, the work of \cite{DT} addresses the general case of unimodular Lie groups equipped with a system of left-invariant vector fields that generate their Lie algebras. Their approach is geometric, establishing a link between the $p=1$ case and a weak isoperimetric inequality, then extending the results to $L^p$ by reducing it to $L^1$. However, their method is specifically designed to prove the inequality involving the gradient (analogue of $\mathcal{L}^{1/2}$) and does not extend to settings involving other powers of the sub-Laplacian $\mathcal{L}$. In contrast, our results accommodate various powers of the sub-Laplacian and utilize the group Fourier transform, refining, and generalizing the work of \cite{DT} in the context of  M\'etivier groups.

As far as we know, Theorem \ref{main result} is new, even for several well-known and extensively studied examples of M\'etivier groups—such as Heisenberg groups, \( H \)-type groups. Unlike the approaches in \cite{CCR} and \cite{CRS}, our method relies on the explicit representation theory of the group, which confines our analysis to the M\'etivier group setting.  We note that the case $p=2$ follows directly from~\eqref{CCR-HPW-stratified} by using the Plancherel formula~\eqref{plan}. It is also worth noting that our method does not readily extend to the \( p > 2 \) case due to technical challenges in estimating Schatten norms. However, the \( p > 2 \) version follows from \eqref{CCR-HPW-stratified} via the dual Hausdorff--Young inequality for the group Fourier transform.

We conclude the introduction with a brief outline of the paper. In Section 2, we introduce the necessary background on harmonic analysis on M\'etivier groups and gather necessary results. Then, in Section 3, we present the proof of our main result.

\section{Preliminaries}\label{prelim}
In this section, we introduce the necessary notations and review fundamental results on harmonic analysis for M\'etivier groups, which are essential for this paper. Most of these definitions and results are drawn from \cite{FS, BKS, BFG, CRS}.
\subsection{Basic Notations}  The letters  $\R,\, \R^+, \,\C$ and $\N$ denote respectively the set of real numbers, positive real numbers, complex numbers, and nonnegative integers. For $1 < p < \infty$, let $p':=\frac p{p-1}$ be the conjugate exponent of $p$. For $p=1$, we define $p'=\infty$ and for $p=\infty$, we define $p'=1$. For a measure space $Y$,  let $L^p(Y)$ denote the usual Lebesgue spaces over $Y$. We denote by $\|f\|_p$ the $L^p$ norm of $f\in L^p(G)$, where $G$ is the group we are working on. Throughout this article, the symbols $ c, C, C_1 $, etc., denote positive constants whose values may change with each occurrence. Everywhere in this article, the notation $f_1 \lesssim f_2$ (respectively, $f_1 \gtrsim f_2$) indicates the existence of positive constants (depending only on the space) $C_1$ and $C_2$ such that $f_1 \leq C_1 f_2$ (respectively, $f_1 \geq C_2 f_2$). We write $f_1 \asymp f_2$ if both $f_1 \lesssim f_2$ and $f_2 \lesssim f_1$ hold. Additionally, we use $C({\varepsilon})$ to denote a constant that depends on the parameter $\varepsilon$. We denote by $\|T\|_{\text{op}}$ the operator norm of the linear operator $T$ on a Banach space $X$ and by $\|\cdot\|_X$ the norm of $X$.
 
\subsection{Two-step nilpotent Lie groups}
A Lie algebra $\mathfrak{g}$ over $\R$ is called two-step nilpotent if $[\g,[\g,\g]]=0$ and $[\g,\g] \neq 0$. The connected, simply connected Lie group $G$ corresponding to such a $\mathfrak{g}$ is called a two-step nilpotent Lie group.

Let $G$ be a connected, simply connected, two-step nilpotent Lie group with the Lie algebra $\g$. We write $\g=\g_1 \oplus \g_2$, where  $[\g_1,\g_1]=\g_2$, and so $\g_2$ is contained in $\mz$, the center   of $\g$. We choose an inner product $\langle\cdot,\cdot\rangle$ in $\g$ such that the above decomposition is orthogonal. Since $G$ is nilpotent, the exponential map $\exp: \g\rightarrow G$ is an analytic diffeomorphism. We therefore identify the elements of $G$ with those of $\g$ via the exponential map. We denote the element $x=\exp(V+Z)\in G$ by $(V,Z)$, where $V\in \g_1$, $Z\in\g_2$. By the  Baker--Campbell--Hausdorff formula, the product law in $G$ is given by
$$(V,Z)(V',Z')=\left(V+V', Z+Z'+\frac12[V,V']\right),\:\:V,V'\in \g_1;\:Z,Z'\in \g_2.$$

We denote by $dV$ and $dZ$ the Lebesgue measure on $\g_1$ and $\g_2$, respectively. Then $dx=dVdZ$ is a Haar measure on $G$.
The Lie algebra $\mathfrak{g}$ is equipped with a canonical family of dilations $\{\delta_r\}_{r>0}$ which are Lie algebra automophisms defined by \cite[p. 5]{FS}
\begin{equation*}
\delta_r\left(V,Z\right)=(rV,r^2Z),\:\:\:V\in\g_1,\:Z\in\g_2.
\end{equation*}
The dilations \( \delta_r \) lift via the exponential map to define a one-parameter group of automorphisms of \( G \), which we still denote by \( \delta_r \), making \( G \) a homogeneous group. We denote by 
\begin{equation*}
Q=\dim \g_1+2\dim \g_2
\end{equation*}
the homogeneous dimension of $G$ and by $e$ the identity element of $G$. The importance of homogeneous dimension stems from the following relation 
\begin{equation}\label{homdim}
\int_Gf(\delta_r(x))\:dx=r^{-Q}\int_Gf(x)\:dx,\quad \text{for all}\:f\in L^1(G),\:\text{and}\:\:r>0.
\end{equation}
 A homogeneous norm on $G$ is a continuous function $|\cdot|:G\to[0,\infty)$ satisfying the following:
\begin{enumerate}
	\item[i)]$|\cdot|$ is smooth on $G\setminus\{e\}$;
	\item[ii)] $|\delta_r(x)|=r|x|$, for all $r>0,\:x\in G$;
	\item [iii)]$|x^{-1}|=|x|$, for all $x\in G$;
	\item[iv)]$|x|=0$ if and only if $x=e$.
\end{enumerate}
It is known that homogeneous norms always exist on homogeneous groups \cite[p. 8]{FS}. It is also known that for any homogeneous norm $|\cdot|$ on $G$ there exists a constant $C>0$ such that
\begin{equation*}
|xy|\leq C(|x|+|y|),\;\:\:\:x\in G, \:y\in G
\end{equation*}
(see \cite[Proposition 1.6]{FS}). Moreover, any two homogeneous norms on $G$ are equivalent: if $|\cdot|_1$ and $|\cdot|_2$ are two homogeneous norms on $G$ then there exists a constant $C>0$ such that 
\begin{equation*}
C^{-1}|x|_1\leq |x|_2\leq C|x|_1,\:\:\:\text{for all}\:\:x\in G.
\end{equation*}
From now onwards, we shall work with a fixed homogeneous norm $|\cdot|$ on $G$. We will need the following formula for integration in ``polar coordinates" \cite[Proposition 1.15]{FS}: for all $f\in L^1(G)$,
\begin{equation}\label{polarcordinate}
\int_Gf(x)\:dx=\int_{0}^{\infty}\int_{S}f(\delta_r(\omega))r^{Q-1}\:d\sigma(\omega)\:dr,
\end{equation}
where $S=\{\omega\in G:|\omega|=1\}$ and $\sigma$ is a unique positive Radon measure on $S$ such that $\sigma(S)=1$. We denote by $\mathcal{S}(G)$ the Schwartz space of $G$, that is, the space of functions $f$ on $G$ such that $f\circ\exp$ is in the Schwartz space of the Euclidean space $\mathfrak{g}$.

\subsection{Harmonic analysis on {M\'etivier} groups} In this subsection, we will describe the representation theory of {M\'etivier}  groups, mostly gathered from \cite{CRS,BFG, BKS} (see also \cite{CG, R}). Let $G$ be a connected, simply connected, two-step nilpotent Lie group and $\g$, $\g_1$, $\g_2$ be as defined in the previous subsection. Let $\g^*,\g_1^*$, and $\g_2^*$ denote the dual vactor spaces of $\g$, $\g_1$, and $\g_2$ respectively. For $\lambda\in \g_2^*$, let $B_{\lambda}$ stand for the skew-symmetric bilinear form on $\g_1$ defined by 
$$B_{\lambda} (V,V')=\lambda([V,V']),\quad (V,V'\in \g_1).$$
\begin{defn}[\cite{Me}]
A connected, simply connected, two-step nilpotent Lie group $G$ is said to be a M\'etivier group if the bilinear form $B_{\la}$ is non-degenerate for all $\la\in\g_2^*\setminus\{0\}$.
\end{defn}
Since for each $\la \in \g_2^*\setminus\{0\}$, $B_\la$ is a skew-symmetric, non-degenerate bilinear form on $\g_1$, it follows that $\dim \g_1 = 2n$ for some $n \in \mathbb{N}$. We choose a basis $\{V_1, \cdots, V_{2n}\}$ of $\g_1$ and a basis $\{T_1, \cdots, T_k\}$ of $\g_2$. Let $\langle\cdot,\cdot\rangle$ be an inner product rendering $\{V_1, \cdots, V_{2n},T_1, \cdots, T_k\}$ an orthonormal basis of $\g$. The inner product $\langle\cdot,\cdot\rangle$ induces a norm on the dual $\g_2^*$, which we denote by $\|\cdot\|$. Let $J_{\lambda}$ be the skew-symmetric endomorphism such that $$B_{\lambda}(V,V')=\langle J_{\lambda}
V,V'\rangle,\quad V,V'\in\g_1.$$ 
Then $G$ is a Métivier group if and only if $J_{\lambda}$ is invertible for all $\la \in\g_2^*\setminus\{0\}$. We call $G$ a Heisenberg type or $H$-type group if $$J_{\lambda}^2=-\|\lambda\|^2 \text{Id}_{\g_1},\quad\text{for all $\lambda\in \g_2^*\setminus\{0\}$}.$$
The family of Heisenberg-type groups forms a proper subclass of Métivier groups; an explicit example of a Métivier group that is not of Heisenberg type can be found, for example, in \cite[Appendix]{MS}. In the following, we describe the basic representation theory of Métivier groups, viewed as a special class of two-step nilpotent Lie groups. For the representation theory of general two-step nilpotent Lie groups, we refer the reader to \cite{CG, CRS, R}.

There exists a Zariski open subset $\Lambda$ of $\g_2^*\setminus\{0\}$ such that for each $\lambda\in \Lambda$ there exists an orthonormal basis $\{P_1(\lambda),\cdots,P_n(\lambda),Q_1(\lambda),\cdots,Q_n(\lambda)\}$ of $\g_1$ and positive numbers $\eta_1(\lambda),\cdots,\eta_n(\lambda)$ satisfying (see \cite[Proposition 3.1]{Ni})
$$\langle J_{\lambda}P_i(\lambda),Q_j(\lambda)\rangle=\delta_{ij}\eta_j(\lambda),\quad 1\leq i,j\leq n.$$
Here, for each $1\leq j\leq n$, the function $\lambda\mapsto \eta_j(\lambda)$ is  homogeneous of degree 1, and continuous on $\g_2^*$, and real analytic on $\Lambda.$ Moreover, each $\eta_j(\la)>0$ for all $\la\in \Lambda$, and $\eta_j\neq \eta_i$ whenever $i\neq j$ for all $\lambda\in \Lambda.$


We note that $\Lambda$ can be taken as a set of full measure in $\R^k$, where $k$ is the dimension of $\g_2$ (see \cite{CRS}). Thus, the homogeneous dimension of $G$ is $Q=2n+2k$.

We now fix $\la \in \Lambda$, and define the following subspaces of $\g_1$ given by
\begin{align*}
    &\p_\la  =\text{span}_{\R}\{P_1(\la),\cdots,P_n(\la)\}, \\
&\q_\la = \text{span}_\R\{Q_1(\la),\cdots,Q_n(\la)\}.
\end{align*}
This gives rise to the following decomposition
$$\g = \g_1 \oplus \g_2 = \p_\la \oplus \q_\la \oplus \mathfrak g_2,$$
and so any element $V\in\g_1$ can be written as
$$V=P(\la)+Q(\la),\:\:\:\text{where}\:P(\la)\in\p_{\la},\:Q(\la)\in\q_{\la}.$$
With respect to the above decomposition, we write any element $x\in G$, as $x=\exp (X(\la,x))$, where $X(\la,x)=(P(\la),Q(\la),T)\in\g$.  More precisely, we identify $x$ with $$(p(\la),q(\la),t):=(p_1(\la),\cdots,p_n(\la),q_1(\la),\cdots,q_n(\la),t_1,\cdots,t_k)\in\R^{2n+k},$$ where
$$P(\la)=\sum_{j=1}^{n}{p_j(\la)P_j(\la)},\:Q(\la)=\sum_{j=1}^{n}{q_j(\la)Q_j(\la)},\:T=\sum_{j=1}^{k}{t_jT_j}.$$
The basis $\{P_1(\la), Q_1(\la), \cdots, P_n(\la), Q_n(\la), T_1, \cdots, T_k \}$ of $\g$ is called an almost symplectic basis. Let $\{T_1^*,\cdots,T_k^*\}$ denote the dual basis in $\g_2^*$.

For $\la\in\Lambda$, we consider the following irreducible unitary representation $\pi_{\la}$ of $G$ realized on $L^2(\p_\la)$ by the following action (See \cite[p. 2693]{BKS}):
\begin{equation}\label{represent}
    \pi_{\la}(x)\phi(\xi)=e^{i\sum_{j=1}^k\la_jt_j+i\sum_{j=1}^n\eta_j(\la)\left(p_j(\la)\xi_j(\la)+\frac{1}{2}p_j(\la)q_j(\la)\right)}\phi(\xi(\la)+q(\la)),
\end{equation}
where $x=(p(\la),q(\la),t)\in G$, $\phi \in L^2(\p_\la)$, $\la=\sum_{j=1}^n\la_jT_j^*$, and $\xi(\la)=\sum_{j=1}^n\xi_j(\la)Q_j(\la)$. There are other irreducible unitary representations of $G$, which do not play any role in the Plancherel formula \cite{CG, R}. To simplify the notation, we will omit the dependence on $\la$ whenever it is clear from the context.

We define the Fourier transform of $f\in L^1(G)$ by the operator-valued integral
$$\F(f)(\la)=\int_Gf(x)\pi_{\la}(x)\:dx,\:\:\:\la\in\Lambda.$$
We note that $\F(f)(\la)$ is a bounded linear operator on $L^2(\p_\la)$ with 
\begin{equation}\label{l1linf}
  \|\F(f)(\la)\|_{\text{op}}\leq \|f\|_{L^1(G)},\:\:\:\text{for all}\:\la\in\Lambda.  
\end{equation}

It is known that if $f\in L^1(G)\cap L^2(G)$, then $\F(f)(\la)$ is a Hilbert--Schmidt operator.   We also have the following Plancherel formula \cite{CRS} (see also \cite{R}): 
\begin{equation}\label{plan}
    \int_{G}|f(x)|^2\:dx=C\int_{\Lambda}\|\F(f)(\la)\|_{S_2(L^2(\mathfrak{p}_{\la}))}^2|\pf(\la)|\:d\la,
\end{equation}
where $|\pf(\la)|=\prod_{j=1}^n\eta_j(\la)$ is the Pfaffian of $B_{\la}$, and $d\la$ is the Lebesgue measure on $\mathfrak{g}_2^*\simeq \R^k$. In the formula above, $\|\cdot\|_{S_p(\mathcal{H})}$, $p\in[1,\infty]$, denotes the Schatten $p$-norm on a separable Hilbert space $\mathcal{H}$. We refer the reader to the Appendix \ref{schatt} for more details on Schatten class operators. The formula \eqref{plan} extends the definition of the Fourier transform to all $f \in L^2(G)$; the Fourier transform thus defined will verify the equality of norms described above.

In view of \eqref{l1linf} and \eqref{plan}, applying the noncommutative Riesz--Thorin interpolation, one can obtain the following analogue of Hausdroff--Young inequality \cite{Kunz}: for $p\in(1,2)$, we have
\begin{equation}\label{hyinq}
    \left(\int_{\Lambda}\|\F(f)(\la)\|_{S_{p'}(L^2(\mathfrak{p}_{\la}))}^{p'}|\pf(\la)|\:d\la\right)^{\frac{1}{p'}}\leq C_p\|f\|_{L^p(G)}.
\end{equation}
The inversion formula for the Fourier transform on $G$ reads as follows \cite[Proposition 1.1]{BFG}: there exists a constant $\kappa>0$ such that for all Schwartz class functions $f$ in $G$ 
\begin{equation}\label{inver}
    f(x)=\kappa \int_{\Lambda}tr\left(\pi_{\la}(x)^*\F(f)(\la)\right)|\pf(\la)|\:d\la,\:\:\:\text{for all}\:x\in G.
\end{equation}
We end this subsection by recording an important property of the group Fourier transform. For $r>0$, we define the standard dilation operator $d_r$  on $L^2(\p_{\la})$ by
$$d_r\phi(\xi)=\phi(r\xi),\:\:\:\xi\in\p_{\la},\:\phi\in L^2(\p_{\la}).$$
For a function $f$ on $G$, we define $\delta_rf(x)=f(\delta_r(x))$, $x\in G$.
\begin{lem}
\label{dila-FT}
 Let $f\in L^1(G)$. Then    
	\begin{equation*}
			\F(\delta_rf)(\lambda) = r^{-Q}  d_r \circ \F(f)(r^{-2} \lambda) \circ d_r^{-1},\:\:\:\la\in\Lambda.
\end{equation*} 
\end{lem}
 \begin{proof}
We first observe that 
	\begin{equation}
		\label{pi-r}
		\pi_{\lambda}(rp,rq,0)=d_r^{-1}\circ\pi_{r^2\lambda }(p,q,0)\circ d_r,\:\:\text{for all}\:p\in\p_{\la},\:q\in\q_{\la},\:r>0,
	\end{equation} which can be easily checked using the definition of $\pi_{\lambda}$ \eqref{represent}, and the fact that each $\eta_j$ is homogeneous of degree one. Now a simple change of variable yields 
\begin{align}\label{ftofdil}
    \F(\delta_rf)(\lambda)&=\int_{G}f(rp,rq,r^2t)\pi_{\lambda}(p,q,t)\:dp\:dq\:dt\nonumber\\
    &=r^{-Q}\int_{p_{\la}}\int_{q_{\la}}\int_{\g_2}f(p,q,t)\pi_{\lambda}(r^{-1}p,r^{-1}q,r^{-2}t)\:dp\:dq\:dt.
\end{align}
For $\mu \in \Lambda$, we define $f^{\mu}$ as the Euclidean (inverse) Fourier transform of $f$ with respect to the central variable evaluated at $\mu$, i.e.,

$$f^{\mu}(p,q)=\int_{\g_2}f(p,q,t)e^{i\mu(t)}\:dt,\:\:(p,q)\in\p_{\mu}\oplus\q_{\mu}.$$
Thus, we can rewrite \eqref{ftofdil} as follows:
\begin{align*}
  \F(\delta_rf)(\lambda)&=r^{-Q}\int_{p_{\la}}\int_{q_{\la}}\int_{\g_2}f(p,q,t)e^{i\lambda(r^{-2}t)}\:dt\:\pi_{\lambda}(r^{-1}p,r^{-1}q,0)\:dq\:dp \\
  &=r^{-Q}\int_{p_{\la}}\int_{q_{\la}}f^{\lambda/r^2}(p,q)\pi_{\lambda}(r^{-1}p,r^{-1}q,0)\:dq\:dp\\
  &=r^{-Q}d_r\circ\left(\int_{p_{\la}}\int_{q_{\la}}f^{r^{-2}\lambda}(p,q)\pi_{r^{-2}\lambda}(p,q,0)\:dq\:dp\right)\circ d_{r^{-1}},
\end{align*}
where we have used \eqref{pi-r} in the last step. The lemma then follows immediately from the definition of $f^{r^{-2}\lambda}$.
 \end{proof}
\subsection{The sub-Laplacian}
We may consider an element $X$ of $\mathfrak{g}$ as a left-invariant differential operator acting on $C^\infty(G)$, where the action is given by
$$X(f)(y) = \left. \frac{d}{dt} f(y \exp (tX)) \right|_{t=0}\quad\quad (y\in G,\:f\in C^\infty(G)).$$
Recall that $\{V_1, V_2,\hdots, V_{2n}\}$ is a basis of $\g_1.$ The sub-Laplacian on $G$ is defined by 
$$\mathcal{L}=-\sum_{j=1}^{2n} V_j^2.$$
It is a positive self-adjoint operator that is homogeneous of degree 2 with respect to dilations, meaning that
 $$\delta_r^{-1}\circ\mathcal{L}\circ\delta_{r}=r^2 \mathcal{L},\quad\text{for all}\:r>0.$$
In order to describe the spectral theory of $\mathcal{L}$,  let us begin with recalling the orthonormal basis of $L^2(\R)$ consisting of one dimensional Hermite functions $\{\varphi_m\}_{m\in\N}$ satisfying 
$$\varphi_m^{\prime\prime}(\tau)-\tau^2\varphi_m(\tau)=-(2m+1)\varphi_m(\tau),\:\:\:\text{for all}\:\tau\in\R.$$
For a fixed $\la\in \Lambda$, we define generalised scaled Hermite functions as follows: 
$$\Phi_{\alpha}^{\eta(\la)}(\xi):= \prod_{j=1}^n\varphi_{\alpha_j,\eta_j(\lambda)}(\xi_j)\quad (\xi=(\xi_1,\xi_2,\cdots,\xi_n)\in \mathbb{R}^n,~\alpha\in \mathbb{N}^n),$$
where 
$$\varphi_{m,\beta}(\tau)=\beta^{\frac{1}{4}}\varphi_m(\beta^{\frac{1}{2}}\tau),\:\:\:(m,\beta)\in\N\times\R^+,\:\tau\in\R.$$
It is known that, for each $\la\in \Lambda$, $\{\Phi_{\alpha}^{\eta(\la)}:\alpha\in\N^n\}$ forms an orthonormal basis of $L^2(\p_{\la})$. The sub-Laplacian $\ml$ satisfies \cite[p. 305]{R}  $$\widehat{\mathcal{L}f}(\lambda)=\widehat{f}(\lambda) H(\eta(\la)),\:\:\:\la\in \Lambda,$$
where $f$ is a smooth function on $G$ and $$H(\eta(\la)):=\sum_{j=1}^n\left(-\frac{\partial^2}{\partial\xi_j^2}+\eta_j(\la)^2\xi_j^2\right),\:\:\:\eta(\la)=(\eta_1(\la),\cdots,\eta_n(\la)).$$ The operator $H(\eta(\la))$ is called the generalized scaled Hermite operator with parameter $\eta(\la)$ whose spectral decomposition is given by the following formula
\begin{equation}\label{scaledher}
    H(\eta(\la))\Phi_{\alpha}^{\eta(\la)}=\sum_{j=1}^n(2\alpha_j+1)\eta_j(\la)\Phi_{\alpha}^{\eta(\la)}.
\end{equation}
We set $$\zeta_j(\alpha,\lambda)=(2\alpha_j+1)\eta_j(\la),\:\:\:(\alpha,\la)\in\N^d\times\Lambda,$$
and the frequencies associated with $H(\eta(\la))$ is defined by
$$\zeta(\alpha,\la)=\sum_{j=1}^n\zeta_j(\alpha,\lambda),\:\:\:(\alpha,\la)\in\N^d\times\Lambda.$$
Thus, for $\psi\in L^2(\p_{\la})$, we have 
\begin{equation}\label{spec-hermite}
    H(\eta(\la)\psi=\sum_{\alpha\in\N^n}\zeta(\alpha,\la)\langle \psi,\Phi_{\alpha}^{\eta(\la)}\rangle \Phi_{\alpha}^{\eta(\la)},
\end{equation}
where the sum on the right-hand side converges in the $L^2$ norm, and 
$$\langle \psi,\Phi_{\alpha}^{\eta(\la)}\rangle=\int_{\p_{\la}}\psi(\xi)\Phi_{\alpha}^{\eta(\la)}(\xi)\:d\xi.$$
As $G$ is M\'etivier, we have $\eta_j(\la)\neq 0$ for all $\la\in \g_2^*\setminus\{0\}$ (see\cite[p. 9]{Ni}). Using the continuity of $\eta_j$, we have the estimate 
$$\eta_{j}(\la)\asymp 1\quad\text{for}~\|\la\|=1\quad\quad (1\leq j\leq n).$$ 
But each $\eta_j$ is homogeneous of degree one. We thus obtain  $$\eta_{j}(\la)\asymp\|\la\|\quad\quad (\lambda\in\g_2^*\setminus\{0\},~1\leq j\leq n).$$ Therefore, recalling the expression for the Plancherel density $|\pf(\la)|=\prod_{j=1}^n\eta_j(\la)$, we have \begin{equation}\label{pfest}
    |\pf(\la)|\asymp\|\la\|^n\quad\quad (\la\in \Lambda).
\end{equation}
Furthermore, this observation, in view of the definition of $\zeta(\alpha,\la)$, allows us to get the following estimate: \begin{equation}
\label{eigenvalueest}
    \zeta(\alpha,\la)\asymp (|\alpha|+n)\|\la\|\quad\quad ((\alpha,\la)\in\N^d\times\Lambda).
\end{equation}

\section{Proof of main results}
This section is dedicated to proving our main result, Theorem \ref{main result}. Throughout this section, unless stated otherwise, all implicit and explicit constants depend only on $\beta, \gamma, n, k$. The proof is divided into two cases: $p = 1$, and $1 < p < 2$. For the reader’s convenience, we present each case as a separate theorem. We begin with the $p = 1$ case.
\subsection{Proof for $p=1$ case}
\begin{thm}\label{l1case}
Suppose that $\gamma>0$ and $\beta>\frac12Q$. Then, for all $f\in L^1(G)$, we have 
$$\|f\|^{\gamma+\beta}_1\lesssim\left(\int_{G}|x|^\gamma |f(x)|dx\right)^\beta~\left( \sup_{\lambda\in \Lambda} \|\mathcal{F}(f)(\lambda) H(\eta(\lambda))^{\frac\beta2}\|_{\text{op}}\right)^\gamma.$$
\end{thm}
 \begin{proof}
 First, we demonstrate that the above inequality remains invariant under dilation and multiplication by a constant. In order to do so, for $r>0$, and $c>0$, we set $$h(x)=c f(\delta_{r^{-1}}(x)),\quad x\in G.$$
 The formula \eqref{homdim} shows that $\|h\|_1=cr^Q\|f\|_1$, and 
 \begin{equation}
     \label{l1dilate}
     \int_{G}|x|^\gamma |h(x)|dx=cr^{Q+\gamma}\int_{G}|x|^\gamma |f(x)|dx.
 \end{equation}
 Now, on the Fourier transform side, using Lemma \ref{dila-FT} we observe that 
 $$\mathcal{F}(h)(\lambda)=c r^Q d_{r^{-1}}\circ \mathcal{F}(f)(r^2\lambda)\circ d_{r^{-1}}^{-1},\quad \lambda\in \Lambda,$$
 which yields 
 \begin{equation}
     \label{dil-spec-hlamda}
     \|\F (h)(\lambda) H(\eta(\lambda))^{\frac\beta2}\|_{\text{op}}=cr^Q\|d_{r^{-1}}\circ \mathcal{F}(f)(r^2\lambda)\circ d_{r^{-1}}^{-1} H(\eta(\lambda))^{\frac\beta2}\|_{\text{op}}.
 \end{equation}
 In view of the spectral decomposition \eqref{spec-hermite} of $H(\eta(\lambda))$, for any $\psi\in L^2(\p_{\la})$, we see that 
 $$H(\eta(r^2\lambda))^{\frac\beta2}\psi=\sum_{\alpha\in \mathbb{N}^n}\zeta(\alpha, r^2\lambda)^{\frac\beta2} \La \psi, \Phi_{\alpha}^{\eta(r^2\lambda)}\Ra \Phi_{\alpha}^{\eta(r^2\lambda)}.$$
Now, using the homogeneity of $\eta$ along with a change of variable, we note that 
\begin{align*}
    \La \psi, \Phi_{\alpha}^{\eta(r^2\lambda)}\Ra \Phi_{\alpha}^{\eta(r^2\lambda)}&=\La \psi, \Phi_{\alpha}^{r^2\eta(\lambda)}\Ra \Phi_{\alpha}^{r^2\eta(\lambda)}\\
    &= \La d_r^{-1}\psi,\Phi_{\alpha}^{\eta(\lambda)}\Ra d_r\Phi_{\alpha}^{\eta(\lambda)}.
\end{align*}
This, together with the homogeneity of $\zeta $ shows that 
\begin{equation}
    \label{dil-Hlambda}
    H(\eta(r^2\lambda))^{\frac\beta2}=r^{\beta} d_{r}\circ H(\eta(\lambda))^{\frac\beta2}\circ d_{r}^{-1}.
\end{equation}
Therefore, plugging \eqref{dil-Hlambda} into \eqref{dil-spec-hlamda}, we obtain 
$$\|\F (h)(\lambda) H(\eta(\lambda))^{\frac\beta2}\|_{\text{op}}=cr^{Q-\beta} \|\F (f)(r^2\lambda) H(\eta(r^2\lambda))^{\frac\beta2}\|_{\text{op}}$$
which yields 
$$\sup_{\lambda\in \Lambda} \|\mathcal{F}(h)(\lambda) H(\eta(\lambda))^{\frac\beta2}\|_{\text{op}}=cr^{Q-\beta}\sup_{\lambda\in \Lambda} \|\mathcal{F}(f)(\lambda) H(\eta(\lambda))^{\frac\beta2}\|_{\text{op}}.$$
This, together with \eqref{l1dilate}, establishes the claimed invariance. 

Thus, in view of this observation, we may assume that 
 \begin{equation}
     \label{l1assmp}
     \|f\|_1=1=\int_{G}|x|^\gamma |f(x)|\:dx.
 \end{equation}
    Hence, in order to prove the theorem, it suffices to prove that 
    \begin{equation}
    \label{mainestl1}
        \sup_{\lambda\in \Lambda} \|\mathcal{F}(f)(\lambda) H(\eta(\lambda))^{\frac\beta2}\|_{\text{op}}\geq C>0.
    \end{equation}
    First, note from \eqref{l1assmp} that for any $a>0$
    \begin{align*}
      1=\int_{G}|x|^\gamma |f(x)|\:dx\geq \int_{|x|\ge a}|x|^\gamma |f(x)|\:dx\geq  a^\gamma \int_{|x|\geq a} |f(x)|\:dx. 
    \end{align*}
This implies that
   \begin{equation}
       \label{l1est1}
       \int_{|x|\leq a}|f(x)|\:dx=\|f\|_1-\int_{|x|\ge a} |f(x)|\:dx\geq 1-a^{-\gamma}.
   \end{equation}
   The constant $a$ will be specified later. 
   We now assert that $f \in L^2(G)$. To do so, we first observe that
   \begin{align}\label{planche}
       &\int_{\Lambda} \|\F (f)(\lambda)\|_{S_2(L^2(\p_{\la}))}^2~ |\pf(\lambda)|~d\lambda\nonumber\\& =\int_{\Lambda} \sum_{\alpha\in \mathbb{N}^n}\|\F (f)(\lambda) \Phi^{\eta{(\lambda)}}_{\alpha}\|_{L^2(\p_{\la})}^2 ~ |\pf(\lambda)|~d\lambda\nonumber\\
       &=\int_{\Lambda} \sum_{\alpha:\zeta(\alpha,\lambda)\leq 1}\|\F (f)(\lambda) \Phi^{\eta{(\lambda)}}_{\alpha}\|_{L^2(\p_{\la})}^2 ~ |\pf(\lambda)|~d\lambda \nonumber\\&\hspace*{2cm}+  \int_{\Lambda} \sum_{\alpha:\zeta(\alpha,\lambda)> 1}\zeta(\alpha,\lambda)^{-\beta} \zeta(\alpha,\lambda)^{\beta}\|\F (f)(\lambda) \Phi^{\eta{(\lambda)}}_{\alpha}\|_{L^2(\p_{\la})}^2  ~ |\pf(\lambda)|~d\lambda.
   \end{align}
   Now, in view of the hypothesis, namely $\|f\|_1=1$, and the inequality \eqref{l1linf}, the last expression is dominated by 
   \begin{align*}
       \int_{\Lambda} \sum_{\alpha:\zeta(\alpha,\lambda)\leq 1} ~ |\pf(\lambda)|~d\lambda+  \sup_{\lambda \in \Lambda} \|\mathcal{F}(f)(\lambda) H(\eta(\lambda))^{\frac\beta2}\|^2_{\text{op}}\int_{\Lambda} \sum_{\alpha:\zeta(\alpha,\lambda)> 1} \zeta(\alpha,\lambda)^{-\beta} ~ |\pf(\lambda)|~d\lambda.
   \end{align*}
 Using the observations that $|\pf(\lambda)|\asymp \|\lambda\|^{n}$ and $\zeta(\alpha,\lambda)\asymp (|\alpha|+n)\|\lambda\|$ (see \eqref{pfest} and \eqref{eigenvalueest}), there exists positive constants $C_1,\:C_2$ such that
\begin{align}\label{smallfreq}
    \int_{\Lambda} \sum_{\alpha:\zeta(\alpha,\lambda)\leq 1} ~ |\pf(\lambda)|~d\lambda \lesssim \sum_{\alpha\in\N^n}\int_{\|\lambda\|\leq \frac{C_1}{|\alpha|+n}}\|\lambda\|^n d\lambda\lesssim\sum_{\alpha\in\N^n} (|\alpha|+n)^{-n-k}<\infty.
\end{align}
Here and throughout the section, we identify $\Lambda$ with $\R^k$ while performing integration over $\Lambda$. Similarly, using the hypothesis $\beta> Q/2=n+k$, we get 
\begin{align}
\label{integralcalc}
   & \nonumber \int_{\Lambda} \sum_{\alpha:\zeta(\alpha,\lambda)> 1} \zeta(\alpha,\lambda)^{-\beta} ~ |\pf(\lambda)|~d\lambda\\ \nonumber  &\lesssim \sum_{\alpha\in\mathbb{N}^n}(|\alpha|+n)^{-\beta} \int_{\|\lambda\|>\frac{C_2}{|\alpha|+n}} \|\lambda\|^{n-\beta}d\lambda\\ 
   &\lesssim\nonumber \sum_{\alpha\in\mathbb{N}^n}(|\alpha|+n)^{-\beta} \int_{\frac{C_2}{|\alpha|+n}}^\infty u^{n-\beta+k-1}
   du\\
   & \lesssim  \sum_{\alpha\in\N^{n}} (|\alpha|+n)^{-n-k} <\infty.
\end{align}
Using the estimates \eqref{smallfreq} and \eqref{integralcalc} in \eqref{planche}, we arrive at 
    \begin{equation*}
        \int_{\Lambda} \|\F (f)(\lambda)\|_{S_2(L^2(\p_{\la}))}^2~ |\pf(\lambda)|~d\lambda\lesssim 1+ \sup_{\lambda\in\Lambda} \|\mathcal{F}(f)(\lambda) H(\eta(\lambda))^{\frac\beta2}\|^2_{\text{op}}.
    \end{equation*}
  Since the right-hand side of the above inequality is finite, applying the Plancherel formula \eqref{plan}, we deduce that $f\in L^2(G).$  We now use H\"older inequality in \eqref{l1est1}, to obtain 
    \begin{align*}
        1-a^{-\gamma}&\leq \int_{|x|\leq a}|f(x)|\:dx
        \leq \left(\int_{|x|\leq a} dx\right)^{\frac12}\left(\int_{|x|\leq a}|f(x)|^2\:dx\right)^{\frac12}
    \end{align*}
    which shows via the integration in ``polar coordinate" formula \eqref{polarcordinate} that 
    \begin{equation}
        \label{l2normlowbdd}
        \|f\|_2^2\geq a^{-Q}(1-a^{-\gamma})^2.
    \end{equation}
 Now, following the approach in \eqref{integralcalc} and carefully tracking the constants, we obtain for any $c>0$ that
    \begin{equation}
        \label{integralcalc2}
        \int_{\Lambda}\sum_{\alpha:\zeta(\alpha,\la)>c}\zeta(\alpha,\la)^{-\beta}|\pf (\la)|\:d\lambda \leq C_3c^{n-\beta+k},
    \end{equation}
    where $C_3>0$ is constant depending only on $n$ and $k$. Consequently, for any $c>0$, this gives
 \begin{align}
 \label{l2lowbddestfourierside}
       & \int_{\Lambda}\sum_{\zeta(\alpha,\lambda)>c} \|\F(f)(\lambda) \Phi^{\eta{(\lambda)}}_{\alpha}\|_{L^2(\p_{\la})}^2 ~ |\pf(\lambda)|~d\lambda\\ \nonumber
       &=  \int_{\Lambda}\sum_{\zeta(\alpha,\lambda)>c}\zeta(\alpha,\la)^{-\frac{\beta}{2}} \|\F(f)(\lambda) H(\eta(\la))^{\frac{\beta}{2}}\Phi^{\eta{(\lambda)}}_{\alpha}\|_{L^2(\p_{\la})}^2 ~ |\pf(\lambda)|~d\lambda\\ \nonumber
       &\leq \sup_{\lambda\in\Lambda}\|\F(f)(\lambda) H(\eta(\la))^{\frac{\beta}{2}}\|^2_{\text{op}}
        \int_{\Lambda}\sum_{\zeta(\alpha,\lambda)>c}\zeta(\alpha,\la)^{-\frac{\beta}{2}} ~ |\pf(\lambda)|~d\lambda\\ \nonumber
        & \leq C_4 c^{n-\beta+k}.
    \end{align}
Finally, combining the last estimate with \eqref{l2normlowbdd}, we conclude that for any $c, a > 0$, we have
 \begin{align*}
    &\int_{\Lambda}\sum_{\alpha:\zeta(\alpha,\lambda)\leq c} \|\F (f)(\lambda) \Phi^{\eta{(\lambda)}}_{\alpha}\|_{L^2(\p_{\la})}^2 ~ |\pf(\lambda)|~d\lambda\\&= \int_{\Lambda}\|\F (f)(\lambda)\|_{S_2(L^2(\p_{\la}))}^2|\pf(\la)|\ d\lambda-\int_{\Lambda}\sum_{\alpha:\zeta(\alpha,\lambda)>c} \|\F (f)(\lambda) \Phi^{\eta{(\lambda)}}_{\alpha}\|_{L^2(\p_{\la})}^2 ~ |\pf(\lambda)|~d\lambda\\
        &= \|f\|_2^2-\int_{\Lambda}\sum_{\alpha:\zeta(\alpha,\lambda)>c} \|\F (f)(\lambda) \Phi^{\eta{(\lambda)}}_{\alpha}\|_{L^2(\p_{\la})}^2 ~ |\pf(\lambda)|~d\lambda\\
        &\geq  a^{-Q}(1-a^{-\gamma})^2-C_4c^{n-\beta+k}.
    \end{align*}
    Since $\beta>Q/2=n+k,$ we can choose $c$ sufficiently large, and $a$ accordingly so that the last quantity is a positive constant.  Therefore, we get positive constants $C_0$ and $c_1$ such that 
    \begin{align*}
    \int_{\Lambda}\sum_{\alpha:\zeta(\alpha,\lambda)\leq c_1} \|\F (f)(\lambda) \Phi^{\eta{(\lambda)}}_{\alpha}\|_{L^2(\p_{\la})}^2 ~ |\pf(\lambda)|~d\lambda \geq C_0.
    \end{align*}
 This, in view of $\|\F f(\la)\|_{\text{op}}\leq \|f\|_1=1$, implies 
 $$ \int_{\Lambda}\sum_{\alpha:\zeta(\alpha,\lambda)\leq c_1} \|\F (f)(\lambda) \Phi^{\eta{(\lambda)}}_{\alpha}\|_{L^2(\p_{\la})} ~ |\pf(\lambda)|~d\lambda \geq C_0.$$
   We now choose $0<c_2<c_1$ suitably so that 
   \begin{align}
   \label{l1lastest}
        \int_{\Lambda}\sum_{\alpha:c_2\leq \zeta(\alpha,\lambda)\leq c_1} \|\F (f)(\lambda) \Phi^{\eta{(\lambda)}}_{\alpha}\|_{L^2(\p_{\la})} ~ |\pf(\lambda)|~d\lambda \geq \frac{C_0}2.
   \end{align} 
   We observe that 
   \begin{align*}
   & \int_{\Lambda}\sum_{\alpha:c_2\leq \zeta(\alpha,\lambda)\leq c_1} \|\F (f)(\lambda) \Phi^{\eta{(\lambda)}}_{\alpha}\|_{L^2(\p_{\la})} ~ |\pf(\lambda)|~d\lambda\\
    & =\int_{\Lambda}\sum_{\alpha:c_2\leq \zeta(\alpha,\lambda)\leq c_1}\zeta(\alpha,\la)^{-\frac{\beta}2} \|\F (f)(\lambda) H(\eta(\lambda))^{\frac{\beta}2}\Phi^{\eta{(\lambda)}}_{\alpha}\|_{L^2(\p_{\la})} ~ |\pf(\lambda)|~d\lambda\\
    &\leq c_2^{-\frac{\beta}{2}} \sup_{\lambda\in\Lambda} \|\mathcal{F}(f)(\lambda) H(\eta(\lambda))^{\frac\beta2}\|_{\text{op}} \int_{\Lambda}\sum_{\alpha:c_2\leq \zeta(\alpha,\lambda)\leq c_1} |\pf(\la)| d\lambda,
   \end{align*}
   where the last integral is finite. This, in view of \eqref{l1lastest}, proves that there exists a positive constant $C(\beta,n,k)$ such that  
$$ \sup_{\lambda\in \Lambda} \|\mathcal{F}(f)(\lambda) H(\eta(\lambda))^{\frac{\beta}{2}}\|_{\text{op}} \geq C(\beta,n,k), $$  
completing the proof of the theorem.
\end{proof}
\subsection{Proof for $1<p<2$ case}
The core idea of the proof is to suitably adapt the approach used in the $p=1$ case. However, since the Schatten norms are more intricate than the operator norm,  its analysis demands a more delicate and nuanced treatment.
To improve readability and facilitate a structured approach to the proof, we first isolate a crucial intermediate step and present it as a lemma. This will help clarify the underlying argument before proceeding to the main proof.
 \begin{lem}
 \label{p12lemma}
     Let $1<p<2$, and $\beta> Q(1/p-1/2).$ Assume that $f\in L^p(G)$ is such that 
     $$A(f,\beta):=\int_{\Lambda}\|\F (f)(\lambda)H(\eta(\lambda))^{\frac{\beta}{2}}\|^{p'}_{S_{p'}(L^2(\p_{\la}))} |\pf(\la)|~d\lambda<\infty.$$ Then for any fixed $r > 0$, the following holds:  
 \begin{equation}
     \label{esteigler}
   \int_{\Lambda}\sum_{\alpha:\zeta(\alpha,\la)\leq r}\|\F (f)(\la)\Phi^{\eta{(\lambda)}}_{\alpha}\|_{L^2(\p_{\la})}^2 |\pf(\la)|~d\lambda \lesssim r^{(n+k)(1-2/p')}\|f\|_p^2,
     \end{equation}
     and 
     \begin{equation}
     \label{esteigger}
         \int_{\Lambda}\sum_{\alpha:\zeta(\alpha,\la)> r}\|\F (f)(\la)\Phi^{\eta{(\lambda)}}_{\alpha}\|_{L^2(\p_{\la})}^2 |\pf(\la)|~d\lambda \lesssim \left(r^{n+k-\beta p/(2-p)}\right)^{1-\frac{2}{p'}} ~A(f,\beta)^2.
     \end{equation}
     Consequently, $f\in L^2(G).$
 \end{lem}
 \begin{proof}
 Fix $r>0.$ Using H\"older's inequality, we observe that 
   \begin{align}\label{hsest}
   &\int_{\Lambda} \sum_{\alpha:\zeta(\alpha,\lambda)\leq r}\|\F (f)(\lambda) \Phi^{\eta{(\lambda)}}_{\alpha}\|_{L^2(\p_{\la})}^2 ~ |\pf(\lambda)|~d\lambda\nonumber\\
       &\leq \left(\int_{\Lambda} \sum_{\alpha:\zeta(\alpha,\lambda)\leq r}\|\F(f)(\lambda) \Phi^{\eta{(\lambda)}}_{\alpha}\|_{L^2(\p_{\la})}^{p'} ~ |\pf(\lambda)|~d\lambda \right)^{\frac2{p'}} \left(\int_{\Lambda} \sum_{\alpha:\zeta(\alpha,\lambda)\leq r} ~ |\pf(\lambda)|~d\lambda  \right)^{1-\frac2{p'}}
   \end{align}
   where the last term, in view of \eqref{pfest}, and \eqref{eigenvalueest}, can be estimated as 
   \begin{align}\label{estpffian}
   \int_{\Lambda} \sum_{\alpha:\zeta(\alpha,\lambda)\leq r} ~ |\pf(\lambda)|~d\lambda\lesssim \int_{\Lambda} \sum_{\alpha:\|\la\|(|\alpha|+n)\leq cr} ~ \|\lambda\|^n~d\lambda= \sum_{\alpha}\int_{\|\la\|\leq \frac{cr}{|\alpha|+n}}\|\la\|^n d\la\lesssim r^{n+k}.
   \end{align}
   Since $p' > 2$, applying Theorem \ref{pgeq2schatten}, we deduce that  
\begin{align*}
       \sum_{\alpha:\zeta(\alpha,\lambda)\leq r}\|\F f(\lambda) \Phi^{\eta{(\lambda)}}_{\alpha}\|_{L^2(\p_{\la})}^{p'}\leq \sum_{\alpha \in\mathbb{N}^n}\|\F f(\lambda) \Phi^{\eta{(\lambda)}}_{\alpha}\|_{L^2(\p_{\la})}^{p'}\leq \|\F f(\la)\|_{S_{p'}({L^2(\p_{\la})})}^{p'}.
   \end{align*}
  which, invoking the Hausdorff--Young inequality \eqref{hyinq}, results in
   \begin{align}\label{otherest}
\left(\int_{\Lambda} \sum_{\alpha:\zeta(\alpha,\lambda)\leq r}\|\F (f)(\lambda) \Phi^{\eta{(\lambda)}}_{\alpha}\|_{L^2(\p_{\la})}^{p'} ~ |\pf(\lambda)|~d\lambda \right)^{\frac2{p'}}  \lesssim \|f\|_p^2.
   \end{align}
  Using the estimates \eqref{estpffian} and \eqref{otherest} in \eqref{hsest} we obtain
   \begin{equation*}
       \label{estI1}
       \int_{\Lambda} \sum_{\alpha:\zeta(\alpha,\lambda)\leq r}\|\F f(\lambda) \Phi^{\eta{(\lambda)}}_{\alpha}\|_{L^2(\p_{\la})}^2 ~ |\pf(\lambda)|~d\lambda \lesssim r^{(n+k)(1-2/p')} \|f\|_p^2,
   \end{equation*}     
   proving \eqref{esteigler}.  Now, to show \eqref{esteigger}, we first write 
   \begin{align}\label{otherl2}
      & \int_{\Lambda}\sum_{\alpha:\zeta(\alpha,\la)> r}\|\F (f)(\la)\Phi^{\eta{(\lambda)}}_{\alpha}\|_{L^2(\p_{\la})}^2 |\pf(\la)|~d\lambda\nonumber\\
      &=\int_{\Lambda} \sum_{\alpha:\zeta(\alpha,\lambda)> r}\zeta(\alpha,\lambda)^{-\beta} \zeta(\alpha,\lambda)^{\beta}\|\F (f)(\lambda) \Phi^{\eta{(\lambda)}}_{\alpha}\|_{L^2(\p_{\la})}^2  ~ |\pf(\lambda)|~d\lambda\nonumber\\
      &\leq\left(\int_{\Lambda} \sum_{\alpha:\zeta(\alpha,\lambda)> r} \zeta(\alpha,\lambda)^{\beta p'/2}\|\F (f)(\lambda) \Phi^{\eta{(\lambda)}}_{\alpha}\|_{L^2(\p_{\la})}^{p'}  ~ |\pf(\lambda)|~d\lambda\right)^{\frac2{p'}}\nonumber\\&\quad\quad\times\left(\int_{\Lambda} \sum_{\alpha:\zeta(\alpha,\lambda)> r}\zeta(\alpha,\lambda)^{-\beta p/(2-p)}  ~ |\pf(\lambda)|~d\lambda\right)^{1-\frac{2}{p'}},
   \end{align}
where we have applied H\"older's inequality to attain the last inequality.  Using \eqref{pfest}, \eqref{eigenvalueest} and Fubini's theorem, there exist constant $c>0$ such that  
   \begin{align}\label{zetapf}
      & \int_{\Lambda} \sum_{\alpha:\zeta(\alpha,\lambda)> r}\zeta(\alpha,\lambda)^{-\beta p/(2-p) }  ~ |\pf(\lambda)|~d\lambda \nonumber\\&\lesssim  \sum_{\alpha\in\mathbb{N}^n}(|\alpha|+n)^{-\beta p/(2-p)}\int_{\|\la\|> \frac{cr}{|\alpha|+n}}\|\la\|^{n-\beta p/(2-p)} d\la \nonumber\\
       & \lesssim r^{n+k-\beta p/(2-p)} \sum_{\alpha\in\mathbb{N}^n} (|\alpha|+n)^{-n-k},
   \end{align}
   where in the second last inequality, the integral is finite  because $$\beta> Q(1/p-1/2)=(n+k)\frac{2-p}{p}.$$ Next, for the other integral, using Theorem \ref{pgeq2schatten}, we see that 
   \begin{align*}
       &\int_{\Lambda} \sum_{\alpha:\zeta(\alpha,\lambda)> r} \zeta(\alpha,\lambda)^{\beta p'/2}\|\F (f)(\lambda) \Phi^{\eta{(\lambda)}}_{\alpha}\|_{L^2(\p_{\la})}^{p'}  ~ |\pf(\lambda)|~d\lambda\\
       &\leq \int_{\Lambda} \sum_{\alpha\in\mathbb{N}^n} \|\F(f)(\lambda)H(\eta(\la))^{\frac{\beta}2} \Phi^{\eta{(\lambda)}}_{\alpha}\|_{L^2(\p_{\la})}^{p'}  ~ |\pf(\lambda)|~d\lambda\\
       &\lesssim \int_{\Lambda}  \|\F f(\lambda)H(\eta(\la))^{\frac{\beta}2}\|_{S_{p'}(L^2(\p_{\la}))}^{p'}  ~ |\pf(\lambda)|~d\lambda.
   \end{align*}
   Using this inequality together with \eqref{zetapf} in \eqref{otherl2} we get
   \begin{align*}
       &\int_{\Lambda}\sum_{\alpha:\zeta(\alpha,\la)> r}\|\F (f)(\la)\Phi^{\eta{(\lambda)}}_{\alpha}\|_{L^2(\p_{\la})}^2 |\pf(\la)|~d\lambda\\
       & \lesssim \left(r^{n+k-\beta p/(2-p)}\right)^{1-\frac{2}{p'}}\left(\int_{\Lambda}  \|\F (f)(\lambda)H(\eta(\la))^{\frac{\beta}2}\|_{S_{p'}({L^2(\p_{\la})})}^{p'}  ~ |\pf(\lambda)|~d\lambda\right)^{\frac2{p'}},
   \end{align*}
   completing the proof of \eqref{esteigler}. Finally, using \eqref{esteigler}, and \eqref{esteigger}, we see that
   \begin{align*}
       &\int_{\Lambda} \|\F (f)(\lambda)\|_{S_{2}(L^2(\p_{\la}))}^2~ |\pf(\lambda)|~d\lambda\\& =\int_{\Lambda} \sum_{\alpha\in \mathbb{N}^n}\|\F (f)(\lambda) \Phi^{\eta{(\lambda)}}_{\alpha}\|_{L^2(\p_{\la})}^2 ~ |\pf(\lambda)|~d\lambda\\
       &=\int_{\Lambda} \sum_{\alpha:\zeta(\alpha,\lambda)\leq r}\|\F (f)(\lambda) \Phi^{\eta{(\lambda)}}_{\alpha}\|_{L^2(\p_{\la})}^2 ~ |\pf(\lambda)|~d\lambda + \int_{\Lambda} \sum_{\alpha:\zeta(\alpha,\lambda)> r}\|\F (f)(\lambda) \Phi^{\eta{(\lambda)}}_{\alpha}\|_{L^2(\p_{\la})}^2  ~ |\pf(\lambda)|~d\lambda\\
       &\lesssim r^{(n+k)(1-2/p')}\|f\|
       _p^2+\left(r^{n+k-\beta p/(2-p)}\right)^{1-\frac{2}{p'}} A (f,\beta)^2<\infty,
   \end{align*}
   by the hypothesis, which, in light of the Plancherel formula \eqref{plan}, establishes that $f \in L^2(G)$.  
\end{proof}
We now proceed with the proof for the case $1<p<2$.
\begin{thm}
    Let $1<p<2$. Suppose that $\gamma>0$, and $\beta>Q(1/p-1/2)$. Then, for all $f\in L^p(G)$, we have 
    \begin{align*}
        \|f\|_p^{\gamma+\beta}\lesssim \||\cdot|^{\gamma}f\|_p^\beta \left( \int_{\Lambda} \| \F (f)(\lambda)H(\eta(\lambda))^{\frac{\beta}{2}}\|^{p'}_{S_{p'}(L^2(\p_{\la}))} |\pf(\lambda)|\:d\lambda\right)^{\frac{\gamma}{p'}}.
    \end{align*}
\end{thm}
\begin{proof}
As in the proof of Theorem \ref{l1case}, we can verify that the desired inequality remains unchanged under dilation and scalar multiplication. Therefore, we may assume that
\begin{align*}
    \||\cdot|^{\gamma}f\|_p=1=\|f\|_p.
\end{align*}

Now, without loss of generality, we may assume that 
\begin{equation}
    \label{assump}
    \int_{\Lambda} \| \F (f)(\lambda)H(\eta(\lambda))^{\frac{\beta}{2}}\|^{p'}_{S_{p'}(L^2(\p_{\la}))} |\pf(\lambda)|\:d\lambda \leq 1,
\end{equation}
as the desired inequality is trivial otherwise. Thus, it is enough to prove that there exists a positive constant $C(\beta,n,k,p)$ such that 
    \begin{equation}
        \label{p12mainclaim}
        \int_{\Lambda} \| \F (f)(\lambda)H(\eta(\lambda))^{\frac{\beta}{2}}\|^{p'}_{S_{p'}(L^2(\p_{\la}))} |\pf(\lambda)|\:d\lambda \geq C(\beta,n,k,p).
    \end{equation}
    Under the assumption that \eqref{assump} holds, by Lemma \ref{p12lemma}, we have $ f \in L^2(G) $. Applying Plancherel formula \eqref{plan} along with \eqref{esteigler} and \eqref{esteigger}, we obtain that for any $ r > 0 $,
   \begin{equation}
   \label{l2estfup}
       \|f\|_2^2\lesssim r^{n+k}(1+r^{-\beta p/(2-p)}).
   \end{equation}
   Now, as in the proof of the Theorem \ref{l1case}, for any $a>0$, we have 
    \begin{align*}
      1=\int_{G}|x|^{\gamma p} |f(x)|^p\:dx\geq \int_{|x|\ge a}|x|^{\gamma p}|f(x)|^p\:dx\geq  a^{\gamma p} \int_{|x|\geq a} |f(x)|^p\:dx 
    \end{align*}
   which yields 
   \begin{equation}
       \label{lpest1}
       \int_{|x|\leq a}|f(x)|^p\:dx=\|f\|^p_p-\int_{|x|\ge a} |f(x)|^p\:dx\geq 1-a^{-\gamma p}.
   \end{equation}

 Now, applying H\"older's inequality, we obtain
   \begin{align*}
       \int_{|x|\leq a} |f(x)|^p\:dx\leq \left( \int_{|x|\leq a} |f(x)|^2\:dx\right)^{\frac{p}2}\left(\int_{|x|\leq a}dx\right)^{1-\frac{p}2}\lesssim a^{Q(1-p/2)}\left( \int_{|x|\leq a} |f(x)|^2\:dx\right)^{\frac{p}2},
   \end{align*}
  which, in view of \eqref{lpest1}, for any $a>0$, then implies 
   \begin{align*}
       \int_{|x|\leq a} |f(x)|^2\:dx \gtrsim \left((1-a^{-\gamma p})a^{Q(p/2-1)}\right)^{\frac{2}{p}}=(1-a^{-\gamma p})^{\frac{2}{p}}a^{Q(1-2/p)}.
   \end{align*}
   This, combined with the Plancherel formula and \eqref{esteigger}, establishes the following spectral estimate on the Fourier transform side: for any $r>0,\:a>0$,
   \begin{align*}
     &  \int_{\Lambda}\sum_{\alpha:\zeta(\alpha, \lambda)\leq r} \|\F(f)(\la)\Phi^{\eta{(\lambda)}}_{\alpha}\|_{L^2(\p_{\la})}^2 ~|\pf (\la)|~d\lambda\\
     &= \|f\|_2^2- \int_{\Lambda}\sum_{\alpha:\zeta(\alpha, \lambda)>r} \|\F (f)(\la)\Phi^{\eta{(\lambda)}}_{\alpha}\|_{L^2(\p_{\la})}^2 ~|\pf (\la)|~d\lambda\\
     &\geq  \int_{|x|\leq a} |f(x)|^2\:dx- \int_{\Lambda}\sum_{\alpha:\zeta(\alpha, \lambda)>r} \|\F (f)(\la)\Phi^{\eta{(\lambda)}}_{\alpha}\|_{L^2(\p_{\la})}^2 ~|\pf (\la)|~d\lambda\\
     &\geq C_1(1-a^{-\gamma p})^{\frac{2}{p}}a^{Q(1-2/p)}- C(n,k,\beta)r^{\frac{Q(2-p)}{2p}-\beta}
   \end{align*}
Since  $\beta>Q(1/p-1/2)$, we note that $\frac{Q(2-p)}{2p}-\beta <0$. Thus, by choosing $r=r_1$ sufficiently large (depending on $\beta,\gamma,Q,p$, if necessary), and $a$ acordingly so that there exists  $C_2>0$ such that 
\begin{align*}
    C_1(1-a^{-\gamma p})^{\frac{2}{p}}a^{Q(1-2/p)}- C(n,k,\beta)r_1^{\frac{Q(2-p)}{2p}-\beta}\geq C_2,
\end{align*}
we get
  \begin{align*}
      \int_{\Lambda}\sum_{\alpha:\zeta(\alpha, \lambda)\leq r_1} \|\F (f)(\la)\Phi^{\eta{(\lambda)}}_{\alpha}\|_{L^2(\p_{\la})}^2 ~|\pf (\la)|~d\lambda \geq C_2.
  \end{align*}
 We recall from \eqref{esteigler} that, there exists $C_3>0$ such that  for any $r>0$
 \begin{equation*}
   \int_{\Lambda}\sum_{\alpha:\zeta(\alpha,\la)\leq r}\|\F (f)(\la)\Phi^{\eta{(\lambda)}}_{\alpha}\|_{L^2(\p_{\la})}^2 |\pf(\la)|~d\lambda \leq C_3 r^{\frac{Q}{2}\frac{2-p}{p}}.
     \end{equation*}
 As $p<2$, we can choose $r_2>0$ small such that 
  $$C_3r_2^{\frac{Q}{2}\frac{2-p}{p}}\leq \frac12 C_2$$ 
  so that, in view of \eqref{esteigger}, we have
  \begin{align*}
       \int_{\Lambda}\sum_{\alpha:\zeta(\alpha, \lambda)\leq  r_2} \|\F (f)(\la)\Phi^{\eta{(\lambda)}}_{\alpha}\|_{L^2(\p_{\la})}^2 ~|\pf (\la)|~d\lambda \leq  \frac12 C_2.
  \end{align*}
  Consequently, 
  \begin{align}
  \label{annularspecest}
       &\int_{\Lambda}\sum_{\alpha:r_2<\zeta(\alpha, \lambda)\leq r_1} \|\F (f)(\la)\Phi^{\eta{(\lambda)}}_{\alpha}\|_{L^2(\p_{\la})}^2 ~|\pf (\la)|~d\lambda \nonumber \\
       &= \int_{\Lambda}\sum_{\alpha: \zeta(\alpha, \lambda)\leq r_1} \|\F (f)(\la)\Phi^{\eta{(\lambda)}}_{\alpha}\|_{L^2(\p_{\la})}^2 ~|\pf (\la)|~d\lambda\nonumber\\&\quad \quad- \int_{\Lambda}\sum_{\alpha:\zeta(\alpha, \lambda)\leq r_2} \|\F (f)(\la)\Phi^{\eta{(\lambda)}}_{\alpha}\|_{L^2(\p_{\la})}^2 ~|\pf (\la)|~d\lambda\geq \frac12 C_2.
  \end{align}

   Now, using Theorem \ref{pgeq2schatten}, we note that 
   \begin{align}
   \label{lowbddest1}
      &\int_{\Lambda} \| \F (f)(\lambda)H(\eta(\lambda))^{\frac{\beta}{2}}\|^{p'}_{S_{p'}(L^2(\p_{\la}))} |\pf(\lambda)|\:d\lambda \nonumber\\
      &\geq \int_{\Lambda} \sum_{\alpha\in \mathbb{N}^n}\| \F (f)(\lambda)H(\eta(\lambda))^{\frac{\beta}{2}}\Phi^{\eta(\la)}_{\alpha}\|^{p'}_{L^2(\p_{\la})} |\pf(\lambda)|\:d\lambda \nonumber\\
      &\geq \int_{\Lambda} \sum_{\alpha:r_2<\zeta(\alpha,\lambda)\leq r_1}\zeta(\alpha,\la)^{\beta p'/2}\| \F(f)(\lambda)\Phi^{\eta(\la)}_{\alpha}\|^{p'}_{L^2(\p_{\la})}|\pf(\lambda)|\:d\lambda \nonumber\\
      & \geq r_2^{\beta p'/2}\int_{\Lambda} \sum_{\alpha:r_2<\zeta(\alpha,\lambda)\leq r_1}\| \F (f)(\lambda)\Phi^{\eta(\la)}_{\alpha}\|^{p'}_{L^2(\p_{\la})} |\pf(\lambda)|\:d\lambda.
   \end{align}
   Applying H\"older's inequality, we get 
   \begin{align}\label{holder2}
      & \int_{\Lambda} \sum_{\alpha:r_2<\zeta(\alpha,\lambda)\leq r_1}\| \F (f)(\lambda)\Phi^{\eta(\la)}_{\alpha}\|^{2}_{L^2(\p_{\la})} |\pf(\lambda)|\:d\lambda\nonumber\\
      &\leq \left(\int_{\Lambda} \sum_{\alpha:r_2<\zeta(\alpha,\lambda)\leq r_1}\| \F (f)(\lambda)\Phi^{\eta(\la)}_{\alpha}\|^{p'}_{L^2(\p_{\la})} |\pf(\lambda)|\:d\lambda\right)^{\frac2{p'}} \left( \int_{\Lambda} \sum_{\alpha:r_2<\zeta(\alpha,\lambda)\leq r_1} |\pf(\lambda)|\:d\lambda\right)^{1-\frac{2}{p'}}.
   \end{align}
  Using the estimate for Pfaffian \eqref{pfest}, one can show that 
   \begin{align*}
       \int_{\Lambda} \sum_{\alpha:r_2<\zeta(\alpha,\lambda)\leq r_1} |\pf(\lambda)|\:d\lambda\lesssim r_1^{n+k}-r_2^{n+k}.
   \end{align*}
  Thus, invoking the above estimate and \eqref{annularspecest} in \eqref{holder2}, we get
   \begin{align*}
       &\int_{\Lambda} \sum_{\alpha:r_2<\zeta(\alpha,\lambda)\leq r_1}\| \F (f)(\lambda)\Phi^{\eta(\la)}_{\alpha}\|^{p'}_{L^2(\p_{\la})} |\pf(\lambda)|\:d\lambda\\
       &\gtrsim (r_1^{n+k}-r_2^{n+k})^{(1-p'/2)}\left(\int_{\Lambda} \sum_{\alpha:r_2<\zeta(\alpha,\lambda)\leq r_1}\| \F (f)(\lambda)\Phi^{\eta(\la)}_{\alpha}\|^2_{L^2(\p_{\la})}|\pf(\lambda)|\:d\lambda\right)^{\frac{p'}{2}}\\
       &\gtrsim (r_1^{n+k}-r_2^{n+k})^{(1-p'/2)}.
   \end{align*}
   This, together with \eqref{lowbddest1}, shows that
   \begin{align*}
       \int_{\Lambda} \| \F (f)(\lambda)H(\eta(\lambda))^{\frac{\beta}{2}}\|^{p'}_{S_{p'}} |\pf(\lambda)|~d\lambda \gtrsim r_2^{\beta p'/2}(r_1^{n+k}-r_2^{n+k})^{(1-p'/2)}
   \end{align*}
   This completes the proof of the theorem.
\end{proof}
\subsection{Concluding remark} The arguments presented above make essential use of explicit estimates for the eigenvalues~\eqref{eigenvalueest} and the associated Plancherel density~\eqref{pfest}. To the best of our knowledge, such estimates are possible only in the setting of M\'etivier groups among two-step nilpotent Lie groups, which explains the restriction considered here. Nevertheless, although two-step MW groups (see~\cite[Definition~2.1]{BKS}) are structurally somewhat different from M\'etivier groups, their representation theory is similar and remains substantially more explicit than that of general two-step nilpotent Lie groups (see, e.g.,~\cite{R}). These considerations suggest that an extension of the Theorem \ref{main result} to this broader class may be possible, but would require additional work to avoid reliance on explicit estimates of the eigenvalues~\eqref{eigenvalueest} and the Plancherel density~\eqref{pfest}. We intend to pursue this direction in future work.

\appendix\section{Schatten class norms}\label{schatt}
We fix a separable Hilbert space $\mathcal{H}$ with inner product $\langle\cdot,\cdot\rangle$ and norm $\|\cdot\|$. Let us first recall the definition of Schatten class operators. 
\begin{defn}
For $1\leq p\leq\infty$, the Schatten $p$-class of $\mathcal H$,
denoted by $S_p(\mathcal H)$, is defined as the family of all compact operators $T$
on $\mathcal H$ whose singular value sequence, that is, the sequence of eigenvalues of $(T^*T)^{1/2}$, $\{s_n(T)\}_{n\in\N}$ belongs to $l^p(\N)$.   
\end{defn}
The class $S_p(\mathcal H)$ equipped with the norm $$\|T\|_{S_p(\mathcal{H})}=\|\{s_n(T)\}\|_{l^p(\N)},$$ 
is a Banach space. In particular, elements of $S_1(\mathcal{H})$ and $S_2(\mathcal{H})$ are known as trace class operators and Hilbert--Schmidt operators, respectively. We also note that $S_{\infty}(\mathcal{H})$ is the space of all compact operators on $\mathcal{H}$ equipped with the operator norm. 

A sequence of functions $\{f_n\}_{n \in \mathbb{N}}$ in $\mathcal{H}$ is a frame for $\mathcal{H}$ if there exist constants $0 < C_1 \leq C_2$ such that  
$$C_1 \| f \|^2 \leq \sum_{n \in \mathbb{N}} | \langle f, f_n \rangle |^2 \leq C_2 \| f \|^2,\quad\text{for all}\: f \in \mathcal{H}.$$  
For a given frame $\{f_n\}_{n \in \mathbb{N}}$, the smallest possible constant $C_2$ is called the upper frame bound.

We present the following characterization of the Schatten classes in terms of frames. The next theorem is taken from \cite[Section 5]{V}.
\begin{thm}\label{pgeq2schatten}
Let $T$ be a compact operator on a separable Hilbert space $\mathcal{H}$ and $2 < p \leq \infty$. Then $T \in S_p(\mathcal{H})$ if and only if 
$${\{ \|T f_n\|\}_n\in \ell^p,}$$  
for every frame $\{f_n\}_{n \in \mathbb{N}}$ of $\mathcal{H}$. Moreover,  
$$\|T\|_{S_p(\mathcal{H})} = \sup \sum_{n \in \mathbb{N}} \|T f_n\|^p,$$  
where the supremum is taken over all frames $\{f_n\}_{n \in \mathbb{N}}$ of $\mathcal{H}$ with an upper frame bound smaller than or equal to $1$.
\end{thm}

\section*{Acknowledgments}  
Pritam Ganguly and Jayanta  Sarkar are supported by the INSPIRE faculty fellowship (Faculty Registration No.: IFA24-MA-207 and Faculty Registration No.:
IFA22-MA-172) from the Department of Science and Technology, Government of India.

\end{document}